\documentclass[12pt,a4paper]{amsart}
\usepackage{mathrsfs}
\usepackage[active]{srcltx}
\usepackage{enumerate}
\usepackage{amsmath,amssymb,xspace,amsthm}

\usepackage{cite}
\usepackage{graphicx}
\usepackage{amscd}
\usepackage{amsfonts}
\usepackage{color}
\usepackage[mathscr]{eucal}
\usepackage{latexsym}

\theoremstyle{definition}
\newtheorem{definition}{Definition}[section]
\newtheorem{theorem}[definition]{Theorem}
\newtheorem{lemma}[definition]{Lemma}
\newtheorem{proposition}[definition]{Proposition}
\newtheorem{corollary}[definition]{Corollary}

\newtheorem{claim}{Claim}
\numberwithin{equation}{section}

\newcommand{\C}{\mathbb{C}}

\newcommand{\N}{\mathbb{N}}
\newcommand{\Z}{\mathbb{Z}}

\def\spn{\mathrm{span}}
\def\Ind{\mathrm{Ind}}

\def\Im{\mathrm{Im}}

\def\ba{\textbf{a}}

\def\fb{\mathfrak{b}}

\def\a{\alpha}

\def\l{\lambda}
\def\ot{\otimes}

\def\Vir{\mathrm{Vir}}

\date{}

\title{New irreducible tensor product modules for the Virasoro algebra}
\author{Xiangqian Guo, Xuewen Liu and Jing Wang}

\begin{document}

\maketitle
\begin{abstract} In this paper, we obtain a class of Virasoro modules by taking
tensor products of the irreducible Virasoro modules
$\Omega(\lambda,\alpha,h)$ defined in \cite{CG}, with irreducible
highest weight modules $V(\theta,h)$ or with irreducible Virasoro
modules Ind$_{\theta}(N)$ defined in \cite{MZ2}. We obtain the
necessary and sufficient conditions for such tensor product modules
to be irreducible, and determine the necessary and sufficient
conditions for two of them to be isomorphic. These modules are not
isomorphic to any other known irreducible Virasoro modules.


\vskip 6pt \noindent {\em 2010  Math. Subj. Class.:} 17B10, 17B20,
17B65, 17B66, 17B68

\vskip 11pt
\end{abstract}

\section{Introduction}
Let $\C, \Z, \Z_+$ and $\N$ be the sets of all complexes, all
integers, all non-negative integers and all positive integers
respectively. The \textbf{Virasoro algebra} $\Vir$ is an infinite
dimensional Lie algebra over the complex numbers $\C$, with the
basis $\{d_i,c\,\, |\,\, i \in \Z\}$ and defining relations
$$
[d_{i},d_{j}]=(j-i)d_{i+j}+\delta_{i,-j}\frac{i^{3}-i}{12}c, \quad
i,j \in \Z,
$$
$$
[c, d_{i}]=0, \quad i \in \Z.
$$
The algebra $\Vir$ is one of the most important Lie algebras both in
mathematics and in mathematical physics, see for example \cite{KR,
IK} and references therein. The representation theory of the
Virasoro algebra has been widely used in many physics areas and
other mathematical branches, for example, quantum physics \cite{GO},
conformal field theory \cite{FMS}, vertex operator algebras
\cite{LL}, and so on.

The theory of weight Virasoro modules with finite-dimensional weight
spaces (called Harish-Chandra modules) is fairly well developed (see
\cite{KR, FF} and references therein). In particular, a
classification of weight Virasoro modules with finite-dimensional
weight spaces was given by Mathieu \cite{M}, and a classification of
irreducible weight Virasoro modules with at least one finite
dimensional nonzero weight space was given in \cite{MZ1}. Later,
many authors constructed several classes of simple
non-Harish-Chandra modules, including simple weight modules with
infinite-dimensional weight spaces (see \cite{CGZ, CM, LLZ, LZ2})
and simple non-weight modules (see \cite{BM, LGZ, LLZ, LZ1, MW, MZ1,
TZ1, TZ2}).

In particular, taking tensor products of known irreducible modules
is an efficient way to construct new irreducible modules and can
help us understand the structures of the original modules. For
example, the tensor products of irreducible highest weight modules
and intermediate series modules were considered first by \cite{Zh}
and the irreducibility of these tensor modules are completely
determined by \cite{CGZ} and \cite{R}. Recently, another class of
tensor products between certain Omega modules defined and studied in
\cite{GLZ, LZ1} and some modules locally finite over a positive part
defined in \cite{MZ2} were studied in \cite{TZ1, TZ2}. The
irreducibilities and isomorphism classes of these modules are
determined.

The purpose of the present paper is to construct new irreducible
non-weight Virasoro modules by taking tensor products of irreducible
Virasoro modules defined in \cite{CG} and \cite{MZ2}. When considering modules
for the W algebra $W(2,2)$ which are free of rank-$1$ when restricted to the $0$
part of the algebra, the authors constructed a class of new irreducible $W(2,2)$-modules $\Omega(\l,\a,h)$ in \cite{CG2}. Since the Virasoro algebra is a natural subalgebra of $W(2,2)$, one can regard $\Omega(\l,\a,h)$ as Virasoro modules. The explicit structures of these Virasoro modules are investigated in \cite{CG}, and it is interesting that many of them remains irreducible as Virasoro modules. 

The $\Vir$-modules $\Omega(\l,\a,h)$ are quite different to and have more complicated structures than the previous modules $\Omega(\mu,b)$ define in \cite{GLZ, LZ1}, although we use similar notations for them. 
For example, the modules $\Omega(\mu,b)$ are only parameterized by two complexes $\mu$ and $b$, while the modules $\Omega(\l,\a,h)$ are parameterized by two complexes $\l, \a$ and an additional polynomial $h(t)\in\C[t]$; the modules $\Omega(\mu,b)$ are free of rank just $1$ over the Cartan subalgebra, while the modules $\Omega(\l,\a,h)$ are free of infinite rank; 
the modules $\Omega(\mu,b)$ are irreducible if and only if $b\neq0$, while the modules $\Omega(\l,\a,h)$ are irreducible if and only if $\deg(h)=1$ and $\a\neq0$; the reducible module $\Omega(\mu,0)$ has a unique submodule which has codimensional $1$, while the submodule structures of the module $\Omega(\l,\a,h)$ are much more complicated when they are reducible; the isomorphisms and automorphisms among the modules $\Omega(\mu,b)$ are almost trivial, while the isomorphisms and automorphisms among the modules $\Omega(\l,\a,h)$ are of various type (see Lemma 3.3 and Theorem 3.4).

In the present paper, we continue to study the Virasoro modules $\Omega(\l,\a,h)$. Our main tasks
are to show that the irreducible ones of the modules $\Omega(\l,\a,h)$ are new Virasoro modules and to
consider the tensor products of the modules $\Omega(\l,\a,h)$ and the modules with locally finite action of the positive part defined in \cite{MZ2}. The organization of this paper is as follows.
In section 2, we recall the definitions of the modules $\Omega(\lambda,\alpha,h)$,
$V(\theta,h)$ and $\Ind_{\theta}(N)$ and some known results from
\cite{CG} and \cite{MZ2}. In section 3, we obtain the irreducibility
of the tensor products $\Omega(\lambda,\alpha,h)\ot V$, where
$V=V(\theta,h)$ or $V=\Ind_{\theta}(N)$. Then we determine the necessary and sufficient
conditions for two irreducible tensor modules to be isomorphic. In
section 4, we compare the tensor products modules with all other
known non-weight irreducible modules and prove that they are new irreducible Virasoro modules
(in particular, the irreducible ones of $\Omega(\l,\a,h)$ are new). At
last in Section 5, we reformulate these modules as modules induced from
irreducible modules over some subalgebras of $\Vir$. 

In our subsequent paper \cite{GLW}, the main results in this paper are generalized to
tensor products of several Omega modules with the module $\Ind_{\theta}(N)$.


\section{Preliminaries}
Let us first recall the definition of the Virasoro modules
$\Omega(\l, h, \a)$, $V(\theta,h)$ and $\Ind_{\theta}(N)$ and some
basic properties of them. Denote by $\C[t, s]$ the polynomial ring
in two variables $t$ and $s$.

\begin{definition}
Fix any $\lambda \in \C^*=\C\setminus\{0\}$, $\a \in \C$ and
$h(t)\in\C[t]$. Let $\Omega(\l, h, \a)=\C[t,s]$ as a vector space
and we define the $\Vir$-module action as follows:
\begin{equation*}
d_m\big(f(t)s^i\big)=
\l^m(s-m)^i\Big(\Big(s+mh(t)-m(m-1)\a\frac{h(t)-h(\a)}{t-\a}\Big)f(t)-m
(t-m\a)f'(t)\Big),
\end{equation*}
\begin{equation*}
c\big(f(t)s^i\big)=0, \;\; \forall\ \ m\in\Z, i\in\Z_+,
\end{equation*}
where $f\in\C[t]$ and $f'(t)$ is the derivative of $f$ with respect
to $t$.
\end{definition}

For convenience, we define the following operators 
\begin{equation}\label{operator}
 F(f)=\frac{h(t)-h(\a)}{t-\a}f(t)-f'(t),\quad G(f)=h(\a)f+tF(f),\
 \forall\ f\in\C[t].
\end{equation}
then the module action on $\Omega(\l,h,\a)$ can be rewritten as
\begin{equation}\label{def}
 d_m\big(f(t)s^i\big)=\l^m(s-m)^i\Big(sf+m G(f)-m^2\a F(f)\Big),\quad\forall\ m\in\Z, i\in\Z_+, f\in\C[t].
\end{equation}

\begin{theorem}[\cite{CG}]\label{CG}
$\Omega(\l,\a,h)$ is simple if and only if $\deg(h)=1$ and
$\alpha\neq 0$.
\end{theorem}

Let $U :=U(\Vir)$ be the universal enveloping algebra of the
Virasoro algebra $\Vir.$ For any $\theta,h\in \mathbb{C},$ let
$I(\theta,h)$ be the left ideal of $U$ generated by the set
$$ \{d_{i}| i>0\}\bigcup \{d_{0}-h\cdot1, c-\theta \cdot 1\}.$$
The Verma module with highest weight $(\theta,h)$ for $\Vir$ is
defined as the quotient module $\bar{V}(\theta,h) :=U/I(\theta,h).$
It is a highest weight module of $\Vir$ and has a basis consisting
of all vectors of the form
$$ d^{k_{-1}}_{-1}d^{k_{-2}}_{-2}\cdots d^{k_{-n}}_{-n}v_{h}, \ k_{-1},k_{-2},\cdots ,k_{-n}\in \mathbb{Z}_{+},n\in \mathbb{N},$$
where $v_{h} = 1+I(\theta,h).$ Any nonzero scalar multiple of
$v_{h}$ is called a highest weight vector of the Verma module. Then
we have the irreducible highest weight module
$V(\theta,h)=\bar{V}(\theta,h)/J,$ where $J$ is the unique maximal
proper submodule of $\bar{V}(\theta,h).$ For the structure of
$V(\theta,h)$, please refer to \cite{FF} or \cite{A}.

Denote by $\Vir_{+}$ the Lie subalgebra of $\Vir$ spanned by all
$d_{i}$ with $i\geq 0.$ For $n\in \mathbb{Z}_{+},$ denote by
$\Vir^{(n)}_{+}$ the Lie subalgebra of $\Vir$ generated by all
$d_{i}$ for $i>n.$ For any $\Vir_{+}$ module $N$ and $\theta \in
\mathbb{C},$ consider the induced module
$\Ind(N):=U(\Vir)\otimes_{U(\Vir^{+})}N,$ and denote by
Ind$_{\theta}(N)$ the module Ind$(N)/(c-\theta)\text{Ind}(N).$ 
These modules are used to give a characterization of the irreducible
$\Vir$-modules such that the action of $d_k$ are locally finite for
sufficiently large $k$.

\begin{theorem}[\cite{MZ2}]\label{MZ2-1}
Assume that $N$ is an irreducible $\Vir_{+}$-module such that there
exists $k\in \N$ satisfying the following two conditions:
\begin{itemize}
\item[(a)] $d_{k}$ acts injectively on $N$;
\item[(b)] $d_{i}N=0$ for all $i>k$.
\end{itemize}
Then for any $\theta\in\C$ the $\Vir $ module $\Ind_{\theta}(N)$ is
simple.
\end{theorem}

\begin{theorem}[\cite{MZ2}]\label{MZ2-2}
Let $V$ be an irreducible $\Vir$ module. Then the following
conditions are equivalent:
\begin{itemize}
\item[(1)] There exists $k\in\N$ such that $V$ is a locally finite $\Vir^{(k)}_{+}$-module;
 \item[(2)] There exists $n\in\N$ such that $V$ is a locally nilpotent $\Vir^{(n)}_{+}$-module;
\item[(3)] Either $V$ is a highest weight module or $V\cong \Ind_{\theta}(N)$ for some $\theta\in \C$, $k\in \N$ and an irreducible $\Vir_{+}$-module $N$
satisfying the conditions $(a)$ and $(b)$ in Theorem \ref{MZ2-1}.
\end{itemize}
\end{theorem}

In the rest of the paper, we will always fix some $\Vir$-module
$\Omega(\l,\a,h)$ with $\l\in\C^*, \a\in\C$ and $h\in\C[t]$, and an
irreducible $\Vir$-module $V$ such that each $d_{k}$ is locally
finite (equivalently, locally nilpotent) on $V$ for any positive
integer $k$ large enough. From Theorem \ref{MZ2-2}, we know that
either $V\cong V(\theta,h)$ for some $\theta,h\in \C$ or $V\cong
\Ind_{\theta}(N)$ as described in Theorem \ref{MZ2-1}. 

\section{Irreducibility of the module $\Omega(\lambda,\alpha,h)\otimes V$}

In this section we will investigate the structure of the Virasoro
module $\Omega(\l,\a,h)\otimes V$ and in particular, we will
determine its irreducibility. The following technique lemma is
similar to Proposition 3.2 of \cite{CG}.

\begin{proposition}\label{technique}
Let $W$ be a subspace of $\Omega(\l,\a,h)\otimes V$ which is stable
under the action of any $d_m$ for $m$ sufficiently large. Take any
$w=\sum\limits_{i=0}^r a_i(t)s^i\otimes v_i\in W$ for some
$a_i(t)\in\C[t]$ and $v_i\in V$, then for any $0\leq j\leq r+2$ we
have
\begin{equation}\label{j=general}
\sum_{i=j-2}^r\Bigg(\binom{i}{j}a_is^{i-j+1}-\binom{i}{j-1}G(a_i)s^{i-j+1}-\binom{i}{j-2}\a
F(a_i)s^{i-j+2}\Bigg)\otimes v_i\in W,
\end{equation}
where the operators $F, G$ are defined in \eqref{operator} and we
make the convention that $\binom{0}{0}=1$ and $\binom{i}{j}=0$
whenever $j>i$ or $j<0$. In particular,
\begin{enumerate}
\item \label{j=0}
when $j=0$, we have $sw\in W$;
%
%
%
\item \label{j=d+1}
when $j=r+1$, we have $G(a_r)\otimes v_r +\a F(a_{r-1})\otimes
v_{r-1}\in W$;

\item \label{j=d+2}
when $j=r+2$, we have $\a F(a_r)\otimes v_{r}\in W$.
\end{enumerate}
\end{proposition}

\begin{proof} The element in \eqref{j=general} is just the
coefficient of $m^j$ if one expands $d_mw$ as a polynomial in $m$.
Then \eqref{j=general} follows by using the Vandermonde's
determinant.
\end{proof}

\begin{theorem}\label{irre} The module $\Omega (\lambda,\alpha,h)\otimes V$ is irreducible if and only if
$\Omega(\l,\a,h)$ is irreducible, or more precisely, if and only if
$\deg(h)=1$ and $\alpha\neq 0$.
\end{theorem}

\begin{proof} We only need to prove the ``if part''. Suppose that
$\Omega(\l,\a,h)$ is irreducible, then by Theorem \ref{CG}, we have
$\deg(h)=1$ and $\alpha\neq 0$. Set $h(t)=\xi t+\eta$ for
convenience.

Let $W$ be a nonzero submodule of $\Omega(\lambda,\alpha,h)\otimes
V$. It is enough to show $W=\Omega(\l,\a,h)\otimes V$. Take any
nonzero element $w=\sum_{i=0}^r a_{i}s^{i}\otimes v_{i}\in W$ with
$a_i\in\C[t], v_{i}\in V$ such that $r\in\Z_+$ is minimal. By
Proposition \ref{technique}, we have $\a F(a_r)\otimes v_r\in W$.
Since $F(a_r)=\xi a_r-a_r'\neq 0$, by the minimality of $r$, we have
$r=0$ and hence $a_0\otimes v_0\in W$. Fix this $v_0$ and we denote
$$X=\{a\in\C[t,s]\ |\ a\otimes v_0\in W\}.$$
By Proposition \ref{technique} \eqref{j=d+1} and \eqref{j=d+2}, we
see that $f\in X\cap\C[t]$ implies that $F(f)=\xi f-f',
G(f)=h(\a)f+\xi tf-tf'\in X\cap\C[t]$, or, equivalently, $f', tf\in
X$. Using this, we can easily deduce that $\C[t]\subseteq X$ from
$0\neq a_0\in X\cap\C[t]$. Now Proposition \ref{technique}
\eqref{j=0} indicates that $X$ is stable under the multiplication by
$s$. Hence $X=\C[t,s]=\Omega(\l,\a,h)$. Now let
$$Y=\{v\in V\ |\ \Omega(\lambda,\alpha,h)\otimes v\in W\}.$$ Again $Y$ is nonzero and
the module action
$$d_{i}(a\otimes v)=d_ia\otimes v+a\otimes d_iv,\ \ \forall\ a\in\C[t,s], v\in Y$$ implies that $Y$
is a submodule of $V$. Hence $Y=V$ and $W=\Omega(\l,\a,h)\otimes V$,
as desired.
\end{proof}

Then we can determine the necessary and sufficient conditions for
two such irreducible modules to be isomorphic. Before doing this, we
first construct some isomorphisms.

Given any $\l,\a_1, \a_2\in\C^*$, $h_1=\xi_1 t+\eta_1, h_2=\xi_2
t+\eta_2\in\C[t]$ with $\a_1\xi_1=\a_2\xi_2\neq0$, we have the
irreducible modules $\Omega(\l,\a_1, h_1)$ and $\Omega(\l,\a_2,
h_2)$. We define the following sequences $\{b_i, i\in\Z_+\}$ of
complex numbers inductively by
$$b_0=1,\ \ b_1=0,\ \ \text{and}\ \ b_{i+1}=ib_i+i(\eta_2-\eta_1)b_{i-1},\ \forall\ i\in\N.$$
Then we have the following sequences of polynomials in the variable
$x$:
\begin{equation}\label{g_n-def}
g_n(x)=\sum_{i=0}^n{n\choose i}b_{n-i}x^i,\ \forall\ n\in\Z_+.
\end{equation}
The following identities can be easily calculated:
\begin{equation}\label{bn-property-1}
g'_n(x)=ng_{n-1}(x),
\end{equation}
\begin{equation}\label{bn-property-2}
\big(g_{n+1}(x)-xg_n(x)\big)-n\big(g_n(x)-xg_{n-1}(x)\big)=(\eta_2-\eta_1)ng_{n-1}(x).
\end{equation}

 Now we can define a linear map via
\begin{equation}\label{phi-def}
\phi:\ \Omega(\l,\a_1, h_1)\ \rightarrow\ \Omega(\l,\a_2, h_2),\
\phi(s^ih_1^n)=s^ig_n(h_2),\ \forall\ n,i\in\Z_+.
\end{equation}

\begin{lemma}\label{iso omega} Let notations as above, then $\phi$ is an isomorphism of $\Vir$-modules.
\end{lemma}

\begin{proof} Denote the operators in \eqref{operator} as $F_i$ and $G_i,
i=1,2$ for corresponding modules, then it is easy to see that, for
all $n\in\Z_+, i=1,2$,
\begin{equation}\label{FGh}
F_i(h_i^n)=\xi_ih_i^n-n\xi_ih_i^{n-1},\ \text{and}\
G_i(h_i^n)=h_i^{n+1}+(\a_i\xi_i-n) h_i^n+n\eta_ih_i^{n-1}.
\end{equation}

\noindent\textbf{Claim 1.} $\phi(\a_1F_1(h_1^n))=\a_2F_2(g_n(h_2))$
and $\phi(G_1(h_1^n))=G_2(g_n(h_2))$ for all $n\in\Z_+$.

This can be verified straightforward. For example, the second
formula follows from \eqref{bn-property-1}, \eqref{bn-property-2}
and the following calculations:
\begin{equation*}\begin{split}
\phi(G_1(h_1^n))= & \phi\big(h_1^{n+1}+(\a_1\xi_1-n) h_1^n+n\eta_1h_1^{n-1}\big)\\
                = & g_{n+1}(h_2)+(\a_2\xi_2-n) g_{n}(h_2)+n\eta_1g_{n-1}(h_2)
\end{split}\end{equation*}
and
\begin{equation*}\begin{split}
G_2(g_n(h_2))
             = & h_2g_n(h_2)+\a_2\xi_2g_n(h_2)-h_2g'_n(h_2)+\eta_2g'_n(h_2)\\
             = & h_2g_n(h_2)+\a_2\xi_2g_n(h_2)-nh_2g_{n-1}(h_2)+n\eta_2g_{n-1}(h_2).\\
\end{split}\end{equation*}
Now we see that
\begin{equation*}\begin{split}
\phi(d_m(s^ih_1^n))= & \phi(\l^m(s-m)^i(sh_1^n+mG_1(h_1^n)-m^2\a_1F_1(h_1^n)))\\
                   = & \l^m(s-m)^is\big(\phi(h_1^n)+m\phi(G_1(h_1^n))-m^2\phi(\a_1F_1(h_1^n))\big)\\
                   = & \l^m(s-m)^i\big(sg_n(h_2)+mG_2(g_n(h_2))-m^2\a_2F_2(g_n(h_2))\big)\\
                   = & d_m(s^ig_n(h_2))=d_m(\phi(s^ih_1^n)).
\end{split}\end{equation*}
That is, $\phi$ is a nonzero homomorphism between the irreducible
modules $\Omega(\l,\a_1, h_1)$ and $ \Omega(\l,\a_2, h_2)$ and hence
an isomorphism.
\end{proof}

\begin{theorem}\label{iso} Let $\lambda_i\in \mathbb{C}^{*}$, $\deg(h_i) = 1$ and $\alpha_i\neq 0$, where $i=1,2$.
Let $V_1, V_2$ be two irreducible modules over $\Vir$ such that the
action of $d_{k}$ is locally finite on both of them for sufficiently
large $k\in\Z_+$. Then $\Omega(\lambda_1,\alpha_1,h_1)\otimes V_1$
and $\Omega(\lambda_2,\alpha_2,h_2)\otimes V_2$ are isomorphic as
$\Vir$ modules if and only if $\l_1=\l_2, \a_1\xi_1=\a_2\xi_2$ and
$V_1 \cong V_2$ as $\Vir$ modules. Moreover, any such isomorphism is
of the form:
$$\phi\ot\tau:\ \ \Omega(\lambda_1,\alpha_1,h_1)\ot V_1\rightarrow \Omega(\l_1,\a_2,h_2)\ot V_2,\ \ f\ot v\mapsto \phi(f)\ot \tau(v),\ \forall\ f\in\C[t,s], v\in V_1,$$
where $\phi$ is defined as in \eqref{phi-def} and $\tau$ is an
isomorphism between $V_1$ and $V_2$.
\end{theorem}

\begin{proof} The sufficiency of the theorem follows from Lemma \ref{iso omega}. We need only
to prove the necessity. Let $\varphi$ be a $\Vir$-module isomorphism
from $\Omega(\lambda_1,\alpha_1,h_1)\otimes V_1$ to
$\Omega(\lambda_2,\alpha_2,h_2)\otimes V_2$. Take a nonzero element
$v\in V_1$. Suppose
  $$ \varphi(1\otimes v) = \sum^{n}_{i=0}a_{i}s^{i}\otimes w_{i},$$
where $a_i\in\C[t], w_{i}\in V_2$ with $a_n\neq0, w_{n} \neq 0$.

\noindent\textbf{Claim 1.} $n=0, \l_1=\l_2$ and
$\a_1\xi_1=\a_2\xi_2$.

There is a positive integer $K$ such that $d_{m}(v) = d_{m}(w_{i}) =
0$ for all $m \geq K$ and $0 \leq i \leq n$. Taking any $m\geq K$,
we have
  $$(\lambda_{1}^{-m-1}d_{m+1}-\lambda_{1}^{-m}d_{m})(1\otimes v) = \big(h_1(t)-2m\xi_1\alpha_1\big)(1\otimes v).$$
Replacing $m$ with another $l\geq K$ and making the difference of
them, we get
$$\Big(\big(\lambda_{1}^{-l-1}d_{l+1}-\lambda_{1}^{-l}d_{l}\big)-\big(\lambda_{1}^{-m-1}d_{m+1}-\lambda_{1}^{-m}d_{m}\big)\Big)(1\otimes v)
 = 2(m-l)\xi_1\alpha_1(1\otimes v).$$
Then applying $\varphi$, we obtain,
  \begin{equation*}
  \begin{split}
  &2(m-l)\xi_1\alpha_1\sum^{n}_{i=0}a_{i}s^{i}\otimes w_{i} \\
  = & \Big(\big(\lambda_{1}^{-l-1}d_{l+1}-\lambda_{1}^{-l}d_{l}\big)-\big(\lambda_{1}^{-m-1}d_{m+1}-\lambda_{1}^{-m}d_{m}\big)\Big)\sum^{n}_{i=0}a_{i}s^{i}\otimes w_{i} \\
  = & \sum^{n}_{i=0}(\lambda_2/ \lambda_1)^{l+1}(s-l-1)^{i}\big(sa_{i}+(l+1)G_2(a_i)-(l+1)^2\a_2 F_2(a_i)\big)\otimes w_i\\
    & \hskip10pt-\sum^{n}_{i=0}(\lambda_2/ \lambda_1)^{l}(s-l)^{i}\big(sa_{i}+lG_2(a_i)-l^2\a_2 F_2(a_i)\big)\otimes w_i \\
    & -\sum^{n}_{i=0}(\lambda_2/ \lambda_1)^{m+1}(s-m-1)^{i}\big(sa_{i}+(m+1)G_2(a_i)-(m+1)^2\a_2 F_2(a_i)\big)\otimes w_i\\
    & \hskip10pt+\sum^{n}_{i=0}(\lambda_2/ \lambda_1)^{m}(s-m)^{i}\big(sa_{i}+mG_2(a_i)-m^2\a_2 F_2(a_i)\big)\otimes w_i.
  \end{split}
  \end{equation*}
Comparing the coefficients of $s^{n+1}\otimes w_n$ in the above
equation, we can deduce that
  $$\big((\lambda_2/\lambda_1)^{l}-(\lambda_2/\lambda_1)^{m}\big)(\lambda_2/\lambda_1-1)a_{n} = 0,\ \forall\ m,l\geq K,$$
forcing $\lambda_1 = \lambda_2.$ Then the previous equation can be
simplified as
  \begin{equation*}
  \begin{split}
  &2(m-l)\xi_1\alpha_1\sum^{n}_{i=0}a_{i}s^{i}\otimes w_{i} \\
  = & \sum^{n}_{i=0}(s-l-1)^{i}\big(sa_{i}+(l+1)G_2(a_i)-(l+1)^2\a_2 F_2(a_i)\big)\otimes w_i\\
    & \hskip10pt-\sum^{n}_{i=0}(s-l)^{i}\big(sa_{i}+lG_2(a_i)-l^2\a_2 F_2(a_i)\big)\otimes w_i \\
    & -\sum^{n}_{i=0}(s-m-1)^{i}\big(sa_{i}+(m+1)G_2(a_i)-(m+1)^2\a_2 F_2(a_i)\big)\otimes w_i\\
    & \hskip10pt+\sum^{n}_{i=0}(s-m)^{i}\big(sa_{i}+mG_2(a_i)-m^2\a_2 F_2(a_i)\big)\otimes
    w_i.
  \end{split}
  \end{equation*}
Regard it as a polynomial in $m, l\geq K$ with coefficients in
$\Omega_2(\l_2,\a_2,h_2)$. If $n\geq 1$, considering the
coefficients of $l^{n+1}$ gives $\alpha_2
F(a_n)=\a_2(\xi_2a_{n}-a'_{n})=0$, contradicting the fact $a_n\neq
0$. So $n = 0.$ Hence we have
\begin{equation*}
 2(m-l)\a_1\xi_1a_{0}\otimes w_{0}  = 2(m-l)\a_2 F_2(a_0)\otimes
 w_0,\ \forall\ m, l\geq K,
\end{equation*}
that is, $\a_1\xi_1 a_0 =\a_2 \xi_2 a_0-\a_2 a_0'$. Therefore we
obtain that $\a_1\xi_1=\a_2\xi_2$ and $a_0\in\C$. Without loss of
generality, we assume that $a_0=1$. Denote $w_0=w$ and
$\l_1=\l_2=\l$ in what follows.

\noindent\textbf{Claim 2.} There exist polynomials $g_j(h_2)$ in
$h_2$ such that $\varphi(h_1^{j}\ot v)=g_{j}(h_2)\ot w,\ \forall\
j\in\Z_+$.

The claim is clear true for $j=0$ with $g_0(h_2)=1$. Now suppose the
claim holds for non-negative integers no larger than some
$j\in\Z_+$, then for $m\geq K$, we have
\begin{equation*}\begin{split}
\varphi(d_m(h_1^j\ot v))= & \varphi\Big(\l^m\big(sh_1^j+mG_1(h_1^j)-m^2\a_1F_1(h_1^j)\big)\ot v\Big)\\
                        = & d_m\varphi(h_1^j\ot v)= d_m(g_j(h_2)\ot w)\\
                        = & \l^m\big(sg_j(h_2)+mG_2(g_j(h_2))-m^2\a_2F_2(g_j(h_2))\big)\ot w.
\end{split}\end{equation*}
Regarding the expressions in the above equation as polynomials in
$m$ and comparing the coefficients of $m^2$ and $m$, we deduce by
\eqref{FGh} that 
\begin{equation}\label{phiF}\begin{split}
    & \varphi(\a_1F_1(h_1^j)\ot v) = \a_1\varphi\Big(\big(\xi_1h_1^j-j\xi_1h_1^{j-1}\big)\ot v\Big)\\
  = & \a_2F_2(g_j(h_2))\ot w =\a_2\big(\xi_2g_j(h_2)-\xi_2g_j'(h_2)\big)\ot w
\end{split}\end{equation}
and
\begin{equation}\label{phiG}\begin{split}
&  \varphi(G_1(h_1^j)\ot v) = \varphi\Big(\big(h_1^{j+1}+(\a_1\xi_1-j) h_1^j+j\eta_1h_1^{j-1}\big)\ot v\Big)\\
                 = & G_2(g_j(h_2))\ot w = \Big(h_2g_j(h_2)+\a_2\xi_2g_j(h_2)-h_2g_j'(h_2)+\eta_2g'_j(h_2)\Big)\ot w,
\end{split}\end{equation}
where $g_j'$ is the derivative of the polynomial $g_j$. Then we see
that $\varphi(h_1^{j+1}\ot v)=g_{j+1}(h_2)\ot w$ for a suitable
polynomial $g_{j+1}$. The claim follows by induction.

\noindent\textbf{Claim 3.} $\varphi(s^ih_1^{j}\ot
v)=s^ig_{j}(h_2)\ot w,\ \forall\ i,j\in\Z_+$, where $g_j$ are
defined as in \eqref{g_n-def}.

From the equations \eqref{phiF} and \eqref{phiG} and noticing
$\a_1\xi_1=\a_2\xi_2$, we deduce
\begin{equation}\label{g_n-1}\begin{split}
g_j'(h_2)=jg_{j-1}(h_2),\ &\
\end{split}\end{equation}
and
\begin{equation}\label{g_n-2}\begin{split}
 g_{j+1}(h_2)-jg_j(h_2)+\eta_1 jg_{j-1}(h_2) = h_2g_j(h_2)-jh_2g_{j-1}(h_2)+\eta_2jg_{j-1}(h_2).
\end{split}\end{equation}
There exists a unique sequence of polynomials $g_j(x)$ in $x$,
$j\in\Z_+$ satisfying \eqref{g_n-1}, \eqref{g_n-2} and $g_0(x)=1$,
which are just those defined in \eqref{g_n-def}, thanks to
\eqref{bn-property-1} and \eqref{bn-property-2}.

Now suppose the claim holds for some $i,j\in\Z_+$. For $m\geq K$, we
consider
\begin{equation*}\begin{split}
\varphi(d_m(s^ih_1^j\ot v))= & \varphi\Big(\l^m(s-m)^i\big(sh_1^j+mG_1(h_1^j)-m^2\a_1F_1(h_1^j)\big)\ot v\Big)\\
                        = & d_m(\varphi(s^ih_1^j\ot v))= d_m(s^ig_j(h_2)\ot w)\\
                        = & \l^m(s-m)^i\big(sg_j(h_2)+mG_2(g_j(h_2))-m^2\a_2F_2(g_j(h_2))\big)\ot w.
\end{split}\end{equation*}
Regard the expressions in the above equation as polynomials in $m$
and comparing the constant terms, we deduce $\varphi(s^{i+1}h_1^j\ot
v)=s^{i+1}g_j(h_2)\ot w$. The assertion follows by induction. 

\noindent\textbf{Claim 4.} There exists a $\Vir$-module isomorphism
$\tau: V_1\rightarrow V_2$ such that $\varphi(f\ot v)=\phi(f)\ot
\tau(v)$ for all $f\in\C[t,s], v\in V_1$, where $\phi:
\Omega(\l,\a_1,h_1)\rightarrow \Omega(\l,\a_2,h_2)$ is the
$\Vir$-module isomorphism defined by \eqref{phi-def}.

Set $\tau(v)=w$ as in the previous arguments. It is obvious that
$\tau$ is a bijective linear map and we have $\varphi(f\ot
v)=\phi(f)\ot \tau(v)$ for all $f\in\C[t,s], v\in V_1$. Applying
$d_m$, we have
\begin{equation*}\begin{split}
\varphi\big(d_m(f\ot v)\big) = & \varphi\big((d_m f)\ot v+ f\ot (d_m v)\big)\\
                             = & \phi(d_m f)\ot \tau(v)+ \phi(f)\ot \tau(d_m v)\\
   ¡¡¡¡¡¡¡¡¡¡¡¡¡¡¡¡= d_m\varphi(f\ot v) = & d_m\big(\phi(f)\ot \tau(v)\big)\\
                   = & \big(d_m\phi(f)\big)\ot\tau(v)+\phi(f)\ot\big(d_m\tau(v)\big).
\end{split}\end{equation*}
Since $\varphi$ and $\phi$ are both $\Vir$-module homomorphisms, we
have $\tau(d_m v)=d_m\tau(v)$, that is, $\tau$ is a $\Vir$-module
isomorphism.%
\end{proof}

Note that the proof of Theorem \ref{iso} is also valid if $V$ is a
$1$-dimensional trivial module. So we obtain a  similar result for
the module $\Omega(\l,\a,h)$ as a corollary.

\begin{corollary} Let $\lambda_i\in \mathbb{C}^{*}$, $\deg(h_i) = 1$ and $\alpha_i\neq 0$.
Then $\Omega(\lambda_1,\alpha_1,h_1)$ and
$\Omega(\lambda_2,\alpha_2,h_2)$ are isomorphic if and only if
$\l_1=\l_2, \a_1\xi_1=\a_2\xi_2$ and the isomorphism are just
nonzero multiples of $\phi$ defined in \eqref{phi-def}.
\end{corollary}
\section{The irreducible tensor modules are new}

In this section we will compare the irreducible tensor modules with
all other known non-weight irreducible Virasoro modules in
\cite{LZ1, LLZ, MZ2, MW} and \cite{TZ1, TZ2}. Note that modules in
\cite{MW} and \cite{TZ1} are special cases of modules in \cite{TZ2}
respectively.

For any $r\in \mathbb{Z}_{+}, l,m\in \mathbb{Z},$ as in \cite{LLZ},
we denote
$$ \omega^{(r)}_{l,m} = \sum^{r}_{i=0}\binom{r}{i}(-1)^{r-i}d_{l-m-i}d_{m+i} \in U(\Vir).$$

\begin{lemma}\label{omega} Let $\Omega(\l,\a,h)$ and $V$ be the irreducible $\Vir$-modules as
before. Then
\begin{enumerate}
\item For any integer $n$, the action of $d_{n}$ on
$\Omega(\lambda,\alpha,h)$ or $\Omega(\l,\a,h)\otimes V$ is not
locally finite.
\item Suppose $h(t)=\xi t+\eta, \xi\neq0$. For any $f(t,s)\in\Omega(\l,\a,h)$, we have
\begin{equation*}\label{omega_Omega}\begin{array}{l}
\omega^{(r)}_{l,m}(f(t,s)) = 0,\ \forall\ l,m, r\in \mathbb{Z}, r >
4,\end{array}\end{equation*}
$$\omega^{(4)}_{l,m}(f(t,s)) =
24\l^l\a^2\big(\xi^2-2\xi\frac{\partial}{\partial t}
+\frac{\partial^2}{\partial t^2}\big)f(t,s-l)\neq0,\ \forall\
l,m\in\Z.$$

\item For any integer $r>4$, there exists $v\in
V$ and $m,l\in \mathbb{Z}$ such that $$\omega^{(r)}_{l,-m}(f\otimes
v) \neq 0,\ \forall\ f\in\C[s,t]\setminus\{0\}.$$
\end{enumerate}
\end{lemma}

\begin{proof} (1). It is clear that $d_n^k(f),
k\in\Z_+$ are linearly independent in $\C[t,s]$ for any $n\in\Z$ and
$f\in\C[t,s]$. So the assertion follows easily.
%

\noindent(2) For any $f(t)\in\C[t]$ and $j\in\Z_+$, by \eqref{def},
we have
\begin{equation*}
\begin{split}
\omega^{(r)}_{l,m}(s^jf)=& \sum^{r}_{i=0}\binom{r}{i}(-1)^{r-i}d_{l-m-i}d_{m+i}(s^jf)\\
= & \sum^{r}_{i=0}\binom{r}{i}(-1)^{r-i}d_{l-m-i}\l^{m+i}(s-m-i)^j\big(sf+(m+i)G(f)-(m+i)^2\a F(f)\big)\\
=& \sum^{r}_{i=0}\binom{r}{i}(-1)^{r-i}\lambda^{l}(s-l)^j\cdot \\
  &\hskip10pt\Big((s-l+m+i)\big(sf+(l-m-i)G(f)-(l-m-i)^2\a F(f)\big)\\
  &\hskip20pt+(m+i)\big(sG(f)+(l-m-i)G^2(f)-(l-m-i)^2\a FG(f)\big)\\
  &\hskip20pt-(m+i)^2\a\big(sF(f)+(l-m-i)GF(f)-(l-m-i)^2\a F^2(f)\big)\Big).
\end{split}
\end{equation*}
Using the following identity
\begin{equation}\label{identity} \sum_{i=0}^r(-1)^{r-i}{r\choose
i}i^j=0,\ \forall\ j, r\in\Z_+ {\rm\ with\ } j<r,
\end{equation}
we can easily deduce that $\omega^{(r)}_{l,m}\big(s^jf(t)\big) = 0$
provided $r>4$ and
\begin{equation*}
\begin{split}
\omega^{(4)}_{l,m}(s^jf) = & \sum^{4}_{i=0}\binom{4}{i}(-1)^{4-i}i^4\l^l(s-l)^j\a^2F^2(f)\\
                         = & 24\l^l(s-l)^j\a^2\big(\xi^2 f-2\xi f'+f''\big).
\end{split}
\end{equation*}
The result of (2) follows from linearity.

\noindent(3). Fix any $r>4$. Take $v$ to be a highest weight vector
in $V$ if $V$ is a highest weight module, otherwise $v$ can be any
nonzero vector in $V.$ From \cite{FF} and \cite{MZ2} we know that
the vectors $v, d_{-2}v, d_{-3}v,\cdots, d_{-r-2}v$ are linearly
independent in $V$ and there exists $K\in\N$ such that these vectors
are annihilated by $d_{m}$ for all $m>K$. For any $l>K$ and $m=r+2$,
we have $\omega^{(r)}_{l,-m}(f)=0$ for any $0\neq f\in\C[s,t]$ by
(2) and
\begin{equation*}
\begin{split}
\omega^{(r)}_{l,-m}(f\otimes v) = & \sum^{r}_{i=0}\binom{r}{i}(-1)^{r-i}d_{l+m-i}d_{-m+i}(f\otimes v)\\
 = & \sum^{r}_{i=0}\binom{r}{i}(-1)^{r-i}d_{l+m-i}(f)\otimes d_{-m+i}(v)\\
\end{split}
\end{equation*}
which is nonzero since $d_{-m+i}v, i=0,1,\cdots, r$ are linearly
independent.
\end{proof}

\begin{theorem}
The $\Vir$-modules $\Omega(\l,\a,h)$ or
$\Omega(\lambda,\alpha,h)\otimes V$ is not isomorphic to any
irreducible module defined in \cite{MZ2, LZ1, LLZ} or in \cite{TZ2}.
\end{theorem}

\begin{proof} For any irreducible modules in \cite{MZ2}, there
exists a positive integer $n$ such that $d_n$ acts locally finitely.
So neither our module $\Omega(\l,\a, h)$ nor $\Omega(\l,\a,h)\otimes
V$ is isomorphic to any irreducible modules constructed in
\cite{MZ2} by Lemma \ref{omega} (1).


Then we consider the irreducible non-weight $\Vir$-module $A_b$
defined in \cite{LZ1}. From the proof of Theorem 9 in \cite{LLZ} or
the argument in the proof of Corollary 4 in \cite{TZ1}, we have
\begin{equation}\label{omega_A}
\omega^{(r)}_{l,m}(A_{b}) = 0,\ \forall\ l, m \in \mathbb{Z}, r\geq
3. \end{equation}
Combining this with Lemma \ref{omega} (2) and (3),
we see easily that $\Omega(\l,\a,h)\not\cong A_b$ and
$\Omega(\l,\a,h)\otimes V\not\cong A_b$.

Now we recall the irreducible non-weight Virasoro modules defined in
\cite{LLZ}. Let $M$ be an irreducible module over the Lie algebra
$\mathfrak{a}_{k} := \Vir_{+}/ \Vir^{(k)}_{+}, k \in \mathbb{N}$
such that the action of $\bar{d_{k}} := d_{k}+\Vir^{(k)}_{+}$ on $M$
is injective, where $\Vir^{(k)}_{+}=\{d_i\ |\ i>k\}$ and
$\Vir_{+}=\spn\{d_i\ |\ i\in\Z_+\}$. For any $\beta\in
\mathbb{C}[t^{\pm1}]\setminus \mathbb{C}$, the $\Vir$-module
structure on $\mathcal{N}(M,\beta) = M\otimes \mathbb{C}[t^{\pm1}]$
is defined by
\begin{equation*}\begin{array}{l}
d_{m}\cdot (v\otimes t^{n}) = (n +
\sum^{k}_{i=0}\frac{m^{i+1}}{(i+1)!}\bar{d_{i}})v\otimes t^{n+m} +v\otimes (\beta t^{m+n}),\\\\
 c\cdot (v\otimes t^{n}) = 0, \ \forall\ m, n \in \mathbb{Z}.
\end{array}\end{equation*} From the computation in $(6.7)$
of \cite{LLZ} we see that
\begin{equation}\label{omega_N(M)}\begin{array}{l}
\omega^{(r)}_{l,m}(\mathcal{N}(M,\beta)) = 0,\ \forall\ l,m \in \mathbb{Z}, r > 2k+2,\\\\
\omega^{(2k+2)}_{l,m}(w\otimes t^{i}) = (2k+2)!(-1)^{k+1}(\bar{d_k}
^2 w)\otimes t^{i+l}\neq 0,\ \forall\ l,m\in\Z.
\end{array}\end{equation}
Combining the first equation of \eqref{omega_N(M)} with Lemma
\ref{omega} (3), we see that $\mathcal{N}(M,\beta)$ is not
isomorphic to $\Omega(\l,\a,h)\ot V$. Similarly, combining the
second equation of \eqref{omega_N(M)} with Lemma \ref{omega} (2), we
see that $\Omega(\l,\a,h)$ is not isomorphic to
$\mathcal{N}(M,\beta)$ provided $k\geq 2$. If $k=1$, we see that $u$
and $\omega^{(4)}_{l,m}(u)$ are linearly independent for any
$u\in\mathcal{N}(M,\beta)$ provided  $l\neq 0$, while
$\omega^{(4)}_{l,m}(1)=24\a^2\xi^2$ in $\Omega(\l,\a,h)$ by Lemma
\ref{omega} (2). We see
$\Omega(\l,\a,h)\not\cong\mathcal{N}(M,\beta)$ in this case.

Finally, we take any irreducible module in \cite{TZ2}, say,
$T=\big(\bigotimes_{i=1}^k\Omega(\mu_i, b_i)\big)\otimes W$, where
$\mu_1,\cdots,\mu_k\in\C^{*}$ are distinct,
$b_1,\cdots,b_k\in\C^{*}$ and $W$ is a $\Vir$-module such that $d_n$
acts locally finitely for sufficiently large $n\in\N$. As vector
spaces $\Omega(\mu_i,b_i)=\C[s_i]$ and the $\Vir$-action on
$\Omega(\mu_i,b_i)$ is given by
$$d_mf(s_i)=\mu_i^m(s_i+mb_i)f(s_i-m), \ \forall\ m\in\Z.$$
Then $T$ is just the tensor product of $\Omega(\mu_1,b_1),
\cdots,\Omega(\mu_k,b_k)$ and $W$. If $W$ is $1$-dimensional and
$k=1$, then this tensor product module is just $\Omega(\mu_1,b_1)$,
which is just a special case of some module $A_b$ we treated
previously. So we assume that $k\geq2$ or $\dim(W)\geq 2$ in the
following.

It was shown in Proposition 7 of \cite{TZ2} that there exist
$l,m\in\Z$ and $w\in W$ such that
$$\omega^{(r)}_{l,-m}(1\otimes\cdots\otimes 1\otimes
w) \neq 0,\ \forall\ r> 4.$$ We remark that this formula holds for
$r>2$ actually. However, it does not matter for us in the  present
argument. In deed, this formula, in whichever version, together with
Lemma \eqref{omega} (2) implies that $\Omega(\l,\a,h)\not\cong
T=\big(\bigotimes_{i=1}^k\Omega(\mu_i, b_i)\big)\otimes W$.

For any nonzero elements $\sum_{i=0}^{n}a_i(t)s^i\otimes v_i\in
\Omega(\l,\a,h)\otimes V$ with $v_i\in V$, there exists $K\in\N$
such that $d_{j}v_i=0$ for all $j>K$ and $i=0,\cdots, n$. Then by
Lemma \ref{omega} (2), we have
\begin{equation}\label{omega_4.1}
\begin{split}
&\omega^{(4)}_{l,m}\sum_{i=0}^{n}a_{i}(t)s^i\otimes v_i
=24\l^l\a^2\sum_{i=0}^{n}(\xi^2a_i(t)-2\xi
a_{i}^{\prime}(t)+a_{i}^{\prime\prime}(t))(s-l)^i\otimes v_i\neq 0,
\end{split}
\end{equation} for all $l-m-4>K$ and $m>K$, and
\begin{equation}\label{omega_5.1}
\omega^{(5)}_{l,m}\sum_{i=0}^{5}a_{i}(t)s^i\otimes v_i=0,\ \forall\
l-m-5>K, m>K.\end{equation}

If $k=1$, write $s=s_1$ for short. Then for any element
$\sum_{i=0}^{n'}s^i\otimes w_i\in \Omega(\mu_1,b_1)\otimes W$ with
$w_i\in W$, there exists $K'\in\N$ such that $d_{j}w_i=0$ for all
$j>K'$ and $i=0,\cdots, n'$. Since $\Omega(\mu_1,b_1)$ is a special
case of $A_b$, using \eqref{omega_A} we get
\begin{equation}\label{omega_4.2}
\omega^{(4)}_{l,m}\sum_{i=0}^{n'}s^i\otimes
w_i=\sum_{i=0}^{n'}\omega^{(4)}_{l,m}s^i\otimes w_i=0,\ \forall\
l-m-4>K', m>K'.\end{equation} The formulas \eqref{omega_4.1} and
\eqref{omega_4.2} indicate that $\Omega(\l,\a,h)\otimes V$ is not
isomorphic to $\Omega(\mu_1,b_1)\otimes W$.

Now suppose that $k\geq 2$. Given any $0\neq w\in W$, there exists
$K''\in\N$ such that $d_{j}w=0$ for all $j\geq K''$. Then for any
$l-m\geq K''+6$ and $m\geq K''$, we have
$$\omega^{(5)}_{l,m}(1\otimes 1\cdots\otimes 1\otimes w)=\sum_{i=0}^{5}{5\choose i}(-1)^{5-i}d_{l-m-i}d_{m+i}(1\otimes 1\cdots\otimes 1)\otimes w.$$
The coefficient of $s_1\otimes s_2\otimes 1\cdots\otimes 1\otimes w$
is
\begin{equation}\label{omega_5.2}
\begin{split}
 \sum_{i=0}^5{5\choose i}(-1)^{5-i}\big(\mu_1^{l-m-i}\mu_2^{m+i}+\mu_1^{m+i}\mu_2^{l-m-i}\big)
 =(\mu_2-\mu_1)^5\Big(\mu_1^{l-m-5}\mu_2^m-\mu_1^{m}\mu_2^{l-m-5}\Big),
\end{split}\end{equation}
which is nonzero for infinitely many $l$ and $m$ with $l>2m+5$ and
$m\geq K''$. Therefore, $\omega^{(5)}_{l,m}(1\otimes 1\cdots\otimes
1\otimes w)\neq0$ for infinitely many $l$ and $m$ with $l>2m+5,
m\geq K''$. This together with \eqref{omega_5.1} shows that
$\Omega(\l,\a,h)\otimes V\not\cong T$. 
\end{proof}

\begin{corollary}  Let $\lambda, \a \in \mathbb{C}^{*}$, $h\in\C[t]$ with $\deg(h) = 1$, and $V$ is an irreducible
$\Vir$-module such that any $d_{k}$ is locally finite on $V$ for
sufficiently large $k$. Then both $\Omega(\lambda,\alpha,h)$ and
$\Omega(\lambda,\alpha,h)\otimes V$ are new non-weight irreducible
$\Vir$ modules.
\end{corollary}

\section{Reformulations of $\Omega(\l,\a,h)\ot V$.}
%

In this section, we will reformulate the tensor modules
$\Omega(\l,\a,h)\ot V$ as certain induced modules when $V$ is an
irreducible highest weight module or an irreducible Whittaker
module. We first recall the definition of the Whittaker modules for
$\Vir$.

Fix any $n\in\Z_+$ and $\ba=(a_{n},\cdots,a_{2n})\in\C$. Recall that
$\Vir_+^{(n-1)}=\spn\{d_i\ |\ i\geq n\}$. We define a
$\Vir_+^{(n-1)}$-module structure on $\C$ as follows:
$$d_i\cdot1=a_i,\ \forall\ n\leq i\leq 2n,\quad\text{and}\quad d_i\cdot 1=0,\ \forall\ i>2n.$$
Now we have the $\Vir$-module:
$$V_{\ba,\theta}=U(\Vir)\otimes_{U(\Vir_+^{(n-1)})}\C/(c-\theta)U(\Vir)\otimes_{U(\Vir_+^{(n-1)})}\C.$$
In the following, we will still write $x\cdot 1$ instead of $x\ot 1$
for any $x\in U(\Vir)$. When $n\geq 1$, $V_{\ba,\theta}$ is just the
Whittaker modules studied in \cite{LGZ} and when $n=0$, the module
$V_{a_0,\theta}$ is the Verma module with highest weight $a_0$ and
central charge $\theta$ (See \cite{KR}).

Fix any $\l\in\C^*$. As in \cite{MW, TZ1}, let
$\fb_{\l,n+1}=\spn\{d_k-\l^{k-n}d_{n} \ |\ k\geq n+1\}$ be the
subalgebra of $\Vir$. For any $\a\in\C, h\in\C[t]$, we can define a
$\fb_{\l,n+1}$-module structure on $\C[t]$ as follows
\begin{equation}\label{b_lambda}
(d_{k}-\lambda^{k-n}d_{n})\circ f= 
(k-n)\l^k\big(G(f)-(k+n)\a F(f)\big)+(a_k-\l^{k-n}a_n)f,
\end{equation}
where we make the convention that $a_k=0$ provided $k>2n$. Denote
this module as $\C[t]_{\a,h,\ba}$.

\begin{proposition}
The $\fb_{\l,n+1}$-module $\C[t]_{\a,h,\ba}$ is irreducible if and
only if $\a\neq0$ and $\deg(h)=1$.
\end{proposition}
\begin{proof} When $\a=0$, then \eqref{b_lambda} becomes
$$(d_{k}-\lambda^{k-n}d_{n})\circ f= (k-n)\l^kG(f)+(a_k-\l^{k-n}a_n)f,$$ and it is easy to see that
$t^i\C[t]$ is a $\fb_{\l,n+1}$-submodule of $\C[t]_{\a,h,\ba}$ for
any $i\in\Z_+$. When $h\in\C$, then \eqref{b_lambda} becomes
$$(d_{k}-\lambda^{k-n}d_{n})\circ f= (k-n)\l^k\big(h(\a)f-tf'+(k+n)\a f'\big)+(a_k-\l^{k-n}a_n)f,$$ and
the subspace of $\C[t]$ consisting of all polynomials with degree no
larger than $i$ is a $\fb_{\l,n+1}$-submodule for any $i\in\Z_+$.
When $\a\neq0$ and $\deg(h)\geq 2$, we have that $F(\C[t])=\{F(f)\
|\ f\in\C[t]\}$ is a proper $\fb_{\l,n+1}$-submodule of
$\C[t]_{\a,h,\ba}$ (cf. Theorem 3.8 of \cite{CG}).

Now suppose that $\a\neq0$ and $h=\xi t+\eta$ for some
$\xi,\eta\in\C$ with $\xi\neq 0$. Let $W$ be a nonzero submodule of
$\C[t]_{\a,h,\ba}$. Take any nonzero $f\in W$, by \eqref{b_lambda},
we have $G(f)-(k+n)\a F(f)\in W$ for all $k\geq n+1$. By the
Vandermonde's determinant, we obtain that $F(f)=\xi f-f'\in W$ and
$G(f)=h(\a)f+t(\xi f-f')\in W$, which implies $f'\in W$ and $t(\xi
f-f')\in W$. Then we can deduce that $1\in W$ by downward induction
on the degree of $f$ and $\C[t]\subseteq W$ by upward induction.
Hence $\C[t]_{\a,h,\ba}$ is an irreducible $\fb_{\l,n+1}$-module.
\end{proof}

Now fix any $\a\neq0$ and $h=\xi t+\eta$ for some $\xi,\eta\in\C$
with $\xi\neq 0$. We can form the induced $\Vir$-module as
$$\Ind_{\theta,\l}(\C[t]_{\a,h,\ba})=U(\Vir)\otimes_{U(\fb_{\l,n+1})}\C[t]_{\a,h,\ba}/(c-\theta)U(\Vir)\otimes_{U(\fb_{\l,n+1})}\C[t]_{\a,h,\ba}.$$
Then we have the following:

\begin{theorem} Let notation as above, then
$\Omega(\l,\a,h)\otimes V_{\ba,\theta}\cong
\Ind_{\theta,\l}(\C[t]_{\a,h,\ba})$.
\end{theorem}

\begin{proof} 
From the PBW Theorem, we see that the module
$\Ind_{\theta,\lambda}(\mathbb{C}[t^{\pm1}]_{\alpha,h,\ba})$ has a
basis
$$A= \{d_{l}^{k_{l}} d_{l+1}^{k_{l+1}} \cdots d_n^{k_{n}}\otimes t^{i} \ \big|\: k_{l},\cdots,k_{n}\in \mathbb{Z}_{+}, l\in \Z, l\leq n, i\in\Z_+\},$$
and the module $\Omega(\l,\a,h)\otimes V(\ba,\theta)$ has a basis
$$B = \{t^{i}s^{k_{n}}\otimes (d_{l}^{k_{l}}d_{l+1}^{k_{l+1}} \cdots d_{n-1}^{k_{n-1}}\cdot 1) \ \big|\ k_{l},\cdots, k_{n}\in \mathbb{Z}_{+},l\in \Z, l\leq n-1, i\in\Z_+\}.$$
Now we can define the following linear map:
$$\begin{array}{cccc}
\phi: & \Ind_{\theta,\l}(\C[t]_{\a,h,\ba}) & \rightarrow & \Omega(\l,\a,h)\otimes V(\ba,\theta)\\\\
      & d_{l}^{k_{l}}\cdots d_n^{k_{n}}\otimes t^{i} & \mapsto & d_{l}^{k_{l}} \cdots d_n^{k_{n}}(t^{i}\otimes1).
\end{array}$$

\begin{claim}
$\phi$ is a $\Vir$-module homomorphism.
\end{claim}

We first have the following observation by \eqref{b_lambda}:
\begin{equation}\label{observ}\begin{split}
(d_j-\l^{j-n}d_n)(t^i\ot 1) & = \big((d_j-\l^{j-n}d_n)\circ t^i\big) \ot 1,\ \ \forall\ j>n.\\
\end{split} \end{equation}
Then for any $x=d_{l}^{k_{l}} d_{l-1}^{k_{l-1}} \cdots d_n^{k_{n}},
l\leq n, i\in\Z_+$ and $j>n$, we have
\begin{equation*}\begin{split}
\phi(xd_{j}\ot t^i) & = \phi(\l^{j-n}xd_n\ot t^i+x\big(d_{j}-\l^{j-n}d_n\big)\ot t^i\big)\\
                    & = \phi\big(\l^{j-n}xd_n\ot t^i+x\ot (d_{j}-\l^{j-n}d_n)\circ t^i\big)\\
                    & = \l^{j-n}xd_n (t^i\ot 1)+ x\big((d_{j}-\l^{j-n}d_n)\circ t^i\ot 1\big)\\
                    & = xd_j(t^i\ot 1).
\end{split}\end{equation*}
Then for any $j\in\Z$, by the PBW Theorem, we can write $d_jx$ as
$$d_jx=\sum_{r=1}^pc_rx_r+\sum_{r=1}^qb_ry_rd_{j_r}, \ \ c_r,
b_r\in\C$$ such that $x_r\ot t^i, y_r\ot t^i\in A$ and $j_r>n$.
%
%
Then we deduce
\begin{equation*}\begin{split}
\phi(d_{j} x\ot t^i) & = \phi\Big(\sum_{r=1}^pc_rx_r\ot t^i+\sum_{r=1}^qb_ry_rd_{j_r}\ot t^i\Big)\\
                   & = \sum_{r=1}^pc_rx_r(t^i\ot 1)+\sum_{r=1}^qb_ry_rd_{j_r}(t^i\ot 1)\\
                   & = d_jx(t^i\ot 1)=d_j\phi(x\ot t^i).
\end{split}\end{equation*}

\begin{claim}
$\phi$ is a $\Vir$-module isomorphism.
\end{claim}

It is clear that $t^i\ot 1\in \Im(\phi)$ for all $i\in\Z_+$. Then
from $d^j_n(t^i\ot 1)\in \Im(\phi)$ we can deduce that $t^is^j\ot
1\in\Im(\phi)$ for all $j\in\Z_+$ inductively. Then applying
elements of the form $d_{l}^{k_{l}} d_{l+1}^{k_{l+1}} \cdots
d_n^{k_{n}}$ on $\Omega(\l,\a,h)\ot 1$, we can deduce that
$\Omega(\l,\a,h)\ot V_{\ba,\theta}\in \Im(\phi)$ and $\phi$ is
surjective.


To obtain the injectivity of $\phi$, we first define a total order
on $B$ as follows
$$ t^{k_{n+1}}s^{k_{n}}\otimes (d_{l}^{k_{l}} \cdots d_{n-1}^{k_{n-1}}\cdot1) < t^{k'_{n+1}}s^{k'_{n}}\otimes (d_{m}^{k'_{m}} \cdots d_{n-1}^{k'_{n-1}}\cdot1)$$
if and only if there exists $ r\leq n+1$ such that $k_{i}=k'_i$ for
all $i<r$ and $k_r<k'_r$, where we have made the convenience that
$k_i=0$ for $i<l$ and $k'_{i}=0$ for $i<m$. By simple computations
we can obtain that
\begin{equation*}
\begin{split}
d_{l}^{k_{l}} \cdots d_{n-1}^{k_{n-1}}d_{n}^{k_{n}}(t^{i}\otimes1)
=t^{i}s^{k_{n}}\otimes (d_{l}^{k_{l}} \cdots d_{n-1}^{k_{n-1}}\cdot1) + \text{ lower terms },\ \forall\ k_{l},\cdots,k_n,i\in\Z_+.\\
\end{split}
\end{equation*}
Since $B$ is a basis of the module $\Omega(\l,\a,h)\ot V$, so is the
set
 $$\phi(A)=\{d_{l}^{k_{l}} \cdots d_{n-1}^{k_{n-1}}d_{n}^{k_{n}}(t^{i}\otimes1)\ |\ k_{l},\cdots ,k_{n}, i\in \Z_{+}, l\leq n\}.$$
Therefore, $\phi$ is injective, as desired.
\end{proof}

\begin{corollary} Suppose that $n\in\N$. The module $\Ind_{\theta,\l}(\C[t]_{\a,h,\ba})$ is
irreducible if and only if $\deg(h)=1, \alpha\neq0$ and
$a_{2n-1}^2+a_{2n}^2\neq0$.
\end{corollary}

\begin{proof}
By the Theorem $5.2$, we know that the module
$\Ind_{\theta,\l}(\C[t]_{\a,h,\ba})$ is irreducible if and only if
$\Omega(\lambda,\alpha,h)\otimes V_{\ba,\theta}$ is irreducible, if
and only if the module $\Omega(\lambda,\alpha,h)$ and
$V_{\ba,\theta}$ are both irreducible. From Theorem $2.2$ we know
that $\Omega(\l,\a,h)$ is simple if and only if $\deg(h)=1$ and
$\alpha\neq 0$. And $V_{\ba,\theta}$ is an irreducible if and only
if $a_{2n-1}\neq0$ or $a_{2n}\neq 0$ by the Theorem $7$ in
\cite{LGZ}. So we complete the proof.
\end{proof}

\noindent X.G.: School of Mathematics and Statistics, Zhengzhou
University, Zhengzhou 450001, Henan, P. R. China. e-mail: {\tt
guoxq@zzu.edu.cn}.

\noindent X.L.: School of Mathematics and Statistics, Lanzhou
University, Lanzhou 730000, Gansu, P. R. China.
and School of Mathematics and Statistics, Zhengzhou
University, Zhengzhou 450001, Henan, P. R. China. Email:
Email: liuxw@zzu.edu.cn

\noindent J.W.: School of Mathematics and Statistics, Zhengzhou
University, Zhengzhou 450001, Henan, P. R. China. e-mail: {\tt
847202401@qq.com}.

\end{document}